\documentclass{article}

\usepackage{amsmath}
\usepackage{amssymb}
\usepackage{latexsym}
\usepackage{theorem}
\usepackage[all]{xypic}

\title{Descent and the $KH$-assembly map}
\author{Paul D.~Mitchener \\
University of Sheffield}

\newenvironment{proof}{\par \noindent{\bf Proof: }}{\hspace{\stretch{1}} $\Box$ \par \mbox{}}
\newcommand{\noproof}{\hspace{\stretch{1}} $\Box$}
\newtheorem{theorem}{Theorem}[section]
\newtheorem{proposition}[theorem]{Proposition}
\newtheorem{lemma}[theorem]{Lemma}
\newtheorem{corollary}[theorem]{Corollary}
{\theorembodyfont{\rmfamily}
\newtheorem{definition}[theorem]{Definition}
\newtheorem{example}[theorem]{Example}
}
\newenvironment{theorem*}{\par \medskip \noindent{\bf Theorem }}{\par \mbox{}}
\newenvironment{lemma*}{\par \medskip \noindent{\bf Theorem }}{\par \mbox{}}

\newcommand{\Hom}{\mathop{Hom}}

\newcommand{\Ob}{\mathop{Ob}}

\newcommand{\R}{{\mathbb R}}

\newcommand{\C}{{\mathbb C}}
\newcommand{\N}{{\mathbb N}}

\newcommand{\id}{\mathop{id}}

\newcommand{\KK}{\mathbb{K} \mathbb{K}}
\newcommand{\KH}{\mathbb{K} \mathbb{H}}
\newcommand{\Sp}{\mathop{Sp}}
\newcommand{\supp}{\mathop{supp}}

\begin{document}

\maketitle

\tableofcontents

\section*{Abstract}

In this article we show that a general notion of descent in coarse geometry can be applied to the study of injectivity of the $KH$-assembly map.  We also show that the coarse assembly map is injective in general for finite coarse $CW$-complexes.

\section{Introduction}
In \cite{Mitch13}, the author drew a general picture of the principle of descent from isomorphism conjectures in coarse geometry to the injectivity of assembly maps in the sense of \cite{DL,WW}.  This technique was applied to the analytic Novikov conjecture and the algebraic $K$-theory assembly map.

Our main purpose in this article is show how homotopy algebraic $K$-theory fits into this picture.  We also look at the stable versions of bivariant algebraic $K$-theory developed in \cite{Gark} to provide versions of the $KH$-isomorphism conjecture with coefficients.

The machinery to carry out these procedures is by and large already developed.  Firstly, there is the general descent machinery already mentioned.  In this article we give a new result saying that {\em any} coarse assembly map is an isomorphism for finite coarse $CW$-complexes.  It is an interesting question whether this could be extended, for example along the lines of \cite{Wright,Yu} to show that any coarse assembly map is an isomorphism for spaces with finite asymptotic dimension.  However, that is a project for another time.

Back to the present, we also use the bivariant algebraic $K$-theory spectra of \cite{Gark2}.  These are defined for algebras; for our purposes we extend the construction to algebroids using the same techniques as used by Joachim to define the $K$-theory of $C^\ast$-categories in \cite{Jo2}.

The properties of bivariant algebraic $K$-theory spectra are then applied to prove that a certain construction along the lines of that in \cite{ACFP} yields a coarse homology theory in the sense of \cite{Mitch6}.  The application of the descent machinery to the maps under consideration then follows.

\section{Coarse Assembly}

We begin by recalling some machinery from coarse geometry.  First, recall that a {\em coarse space} is a set $X$ along with a collection of priveged subsets $M\subseteq X\times X$ called {\em controlled sets}.  The collection of controlled sets is called a {\em coarse structure} on $X$, and is required to satisfy certain axioms.

To be precise, the collection of controlled sets is required to contain the diagonal, $\Delta_X \subseteq X\times X$.\footnote{This axiom is dropped for {\em non-unital} coarse speces; see \cite{Luu}.}  The controlled sets must be closed under finiste unions, taking subsets, reflections in the diagonal, and composition.  Here, by the {\em composition} of subsets $M_1,M_2 \subseteq X\times X$, we mean the set
\[ M_1 M_2 = \{ (x,z)\in X\times X \ |\ (x,y)\in M_1 \ (y,z)\in M_2 \textrm{ for some }y\in X
\} .  \]
 
We refer the reader to \cite{Roe6}, for example, for further
details.

Given a controlled set $M\subseteq X\times X$, and a subset
$S\subseteq X$, we write
$$M[S] = \{ y\in X \ |\ (x,y)\in M \textrm{ for some }x\in S \}$$

For a point $x\in X$, we write $M_x = M[\{ x \}]$.

If $X$ is a coarse space, and $f,g\colon S\rightarrow X$ are maps
into $X$, the maps $f$ and $g$ are termed {\em close} or {\em
coarsely equivalent} if the set $\{ (f(s) , g(s)) \ |\ s\in S \}$ is
controlled.  We call a subset $B\subseteq X$ {\em bounded} if the
inclusion $B\hookrightarrow X$ is close to a constant map, or
equivalently $B=M_x$ for some controlled set $M$ and some point
$x\in X$.

We also need the notion of a topological space with a compatible coarse structure.  To be precise, if
$X$ is a Hausdorff space, we call it a {\em coarse topological space} if it has a coarse structure where every controlled set is contained in an open controlled set (with the usual product topology on $X\times X$), and the closure of any bounded
set is compact.

If $X$ is a coarse topological space, we say the coarse structure is {\em compatible} with the topology.

The most important examples of coarse spaces to us are the following.  The first is standard.

\begin{example}
If $X$ is a proper metric space.  Equip $X$ with a coarse structure defined by defining the controlled
sets to be subsets of {\em neighbourhoods of the diagonal}:
$$N_R = \{ (x,y) \in X\times X \ |\ d(x,y)< R \}$$

Then $X$ is a coarse topological space.
\end{example}

Our second example comes from \cite{Mitch13}.

\begin{example} \label{compactification}
Let $X$ be a coarse topological space.  Suppose that $X$ is
a topologically dense subset of a Hausdorff space $\overline{X}$.  Call the coarse structure already defined on the space $X$ the {\em ambient coarse structure}, and set $\partial X = \overline{X}\backslash X$.  

Call an open subset $M\subseteq X\times X$ {\em strongly controlled} if:

\begin{itemize}

\item The set $M$ is controlled with respect to the ambient coarse
structure on $X$.

\item Let $\overline{M}$ be the closure of the set $M$ in the space
$\overline{X}$.  Then $\overline{M} \cap (\overline{X}\times \partial X) \cup (\partial X
\times \overline{X} )$ is contained in the diagonal of $\partial X$.

\end{itemize}

We define the {\em continuously controlled coarse structure} with respect to
$\overline{X}$ by saying that the controlled sets are composites of subsets
of strongly controlled open sets.

We write $X$ to denote the space $X$ with its ambient coarse
structure, and $X^\mathrm{cc}$ to denote the space $X$ with the new
continuously controlled coarse structure. 
\end{example}
 
It is shown in \cite{Mitch13} that $X^\mathrm{cc}$ is a coarse topological space.  Our next few examples are certain standard constructions of coarse spaces.

\begin{example}
Let $X$ be a coarse space, let $\sim$ be an equivalence relation on
$X$, and let $X/\sim$ be the set of equivalence classes.  Let $\pi
\colon X\rightarrow X/\sim$ be the quotient map sending each point
$x\in X$ to its equivalence class, $\pi (x)$.

We define the {\em quotient coarse structure} on $X/\sim$ by saying
a subset $M\subseteq X/\sim \times X/\sim$ is controlled if and only
if $M = \pi [M']$ for some controlled set $M'\subseteq
X\times X$.
\end{example}

\begin{example}
Let $X$ and $Y$ be coarse spaces.  Then we define the {\em product},
$X\times Y$ to be the Cartesian product of the sets $X$ and $Y$
equipped with the coarse structure defined by saying a subset
$M\subseteq (X\times Y)\times (X\times Y)$ is controlled if and only if we have controlled sets $M_1\subseteq X\times X$ and $M_2\subseteq Y\times Y$ such that
\[ M\subseteq \{ (u,v,x,y) \ |\ (u,x)\in M_1,\ (v,y)\in M_2 \} . \]
\end{example}

\begin{example}
Let $\{ X_i \ |\ i\in I \}$ be a collection of coarse spaces.  Then,
as a set, the {\em coarse disjoint union}, $\vee_{i\in I} X_i$ is
the disjoint union of the sets $X_i$.

A subset $M\subseteq (\vee_{i\in I} X_i)\times (\vee_{i\in
I} X_i)$ is controlled if it is a subset of a union of the form
\[ \left( \bigcup_{i\in I} M_i \right) \cup
\left( \bigcup_{i,j\in I} B_i \times B_j \right) \]
where each set $M_i\subseteq X_i \times X_i$ is controlled, and each
$B_j\subseteq X_j$ is bounded.
\end{example}

The following definition comes from \cite{Mitch4}.

\begin{definition}
Let $R$ be the topological space $[0,\infty )$ equipped with a coarse structure
compatible with the topology.  We call the space $R$ a {\em
generalised ray} if the following conditions hold.

\begin{itemize}

\item Let $M,N\subseteq R\times R$ be controlled sets.  Then the sum
\[ M+N = \{ (u+x,v+y)\ |\ (u,v)\in M,\ (x,y)\in N \} \]
is controlled.

\item Let $M\subseteq R\times R$ be a controlled set.  Then the set
\[ M^s = \{ (u,v)\in R\times R\ |\ x\leq u,v\leq y,\ (x,y)\in M \}
\]
is controlled.

\item Let $M\subseteq R\times R$ be a controlled set, and $a\in R$.
Then the set
\[ a+M = \{ (a+x,a+y)\ |\ (x,y)\in M \} \]
is controlled.

\end{itemize}

\end{definition}

The classic example of a generalised ray is the metric space $\R_+$ equipped with the bounded coarse structure.  There are others.

Now, a map between coarse spaces is called a {\em coarse map} if the image of a controlled set is controlled, and the pre-image of a bounded set is bounded.  The {\em coarse category} is the category of all coarse spaces and coarse maps.

We call a coarse map $f\colon X\rightarrow Y$ a {\em coarse equivalence} if there is a
coarse map $g\colon Y\rightarrow X$ such that the composites $g\circ
f$ and $f\circ g$ are close to the identities $1_X$ and $1_Y$
respectively.

Coarse spaces $X$ and $Y$ are said to be {\em coarsely equivalent} if there is a coarse equivalence between
them.

The notion of {\em coarse homotopy} is a weakening of this notion.  To be precise, we have the following.

\begin{definition}
Let $X$ be a coarse space equipped with a coarse map
$p\colon X\rightarrow R$.  Then we define the {\em $p$-cylinder of $X$}:
$$I_p X = \{ (x,t)\in X\times R \ |\ t\leq p(x)+1 \}$$

We define coarse maps $i_0,i_1\colon X\rightarrow I_pX$ by the
formulae $i_0 (x) = (x,0)$ and $i_1 (x)=(x,p(x)+1)$ respectively.
\end{definition}

\begin{definition}
Let $f_0,f_1\colon X\rightarrow Y$ be coarse maps.  An {\em elementary coarse homotopy}
between $f_0$ and $f_1$ is a coarse map
$H\colon I_p X\rightarrow Y$ for some $p\colon X\rightarrow R$ such that $f_0 = H\circ i_0$
and $f_1 = H\circ i_1$.

More generally, we call the maps $f_0$ and $f_1$ {\em coarsely
homotopic} if they can be linked by a finite sequence of elementary coarse
homotopies.
\end{definition}

Note that close coarse maps are always coarsely homotopic.

A coarse map $f\colon X\rightarrow Y$ is termed a {\em coarse
homotopy equivalence} if there is a coarse map $g\colon Y\rightarrow
X$ such that the compositions $g\circ f$ and $f\circ g$ are coarsely
homotopic to the identities $1_X$ and $1_Y$ respectively.

Before we introduce the main piece of machinery from \cite{Mitch13}, we need two more technical notions.

\begin{definition}
Let $X$ be a coarse space.  Then we call a decomposition $X=A\cup B$
{\em coarsely excisive} if for every controlled set $m\subseteq X\times X$ there is a controlled set
$M\subseteq X\times X$ such that $m(A)\cap m(B)\subseteq M(A\cap B)$.
\end{definition}

\begin{definition}
We call a coarse space $X$ {\em flasque} if there is a map
$\tau \colon X\rightarrow X$ such that:

\begin{itemize}

\item Let $B\subseteq X$ be bounded.  Then there exists $N\in \N$
such that $\tau^n[X]\cap B = \emptyset$ for all $n\geq N$.

\item Let $M\subseteq X\times X$ be controlled.  Then the union
$\bigcup_{n\in \N} \tau^n [M]$ is controlled.

\item The map $\tau$ is close to the identity map.

\end{itemize}

\end{definition}

A generalised ray is clearly flasque.  We are now ready for the main construction of this section.

\begin{definition}
We call a functor, $E$, from the coarse category to the category of
spectra {\em coarsely excisive} if the following conditions hold.

\begin{itemize}

\item The spectrum $E(X)$ is weakly contractible whenever the coarse
space $X$ is flasque.

\item The functor $E$ takes coarse homotopy equivalences to weak homotopy
equivalences of spectra.

\item For a coarsely excisive decomposition $X=A\cup B$ we have a homotopy
push-out diagram
\[ \begin{array}{ccc}
E(A\cap B) & \rightarrow & E(A) \\
\downarrow & & \downarrow \\
E(B) & \rightarrow & E(X) \\
\end{array} . \]

\end{itemize}

\end{definition}

Now, let $X$ be a coarse topological space.  Then we define the {\em open square}, ${\mathcal S}X$, to be the space $X\times [0,1)$ equipped with the continuously controlled coarse structure arising from considering ${\mathcal S}X$ as a dense subset of $X\times [0,1]$.  It is shown in \cite{Mitch13} that the open square ${\mathcal S}X$ is always flasque.

We define the {\em open} and {\em closed cones} to be the quotients
\[ {\mathcal O}X = \frac{{\mathcal S}X}{X\times \{ 0 \} } \qquad 
 {\mathcal C}X = \frac{X\times [0,1]}{X\times \{ 0 \} } \]
respectively.

The following result also comes from \cite{Mitch13}.

\begin{lemma} \label{M13}
Let $X$ be a coarse topological space, and let $E$ be a coarsely excisive functor.  Then we have a natural weak fibration of spectra
\[ E(X)\rightarrow E({\mathcal C}X) \rightarrow E({\mathcal O}X ) . \]

Further, if the space $X$ is flasque, then so is the closed cone ${\mathcal C}X$.
\noproof
\end{lemma}

\begin{definition}
The boundary map $\partial \colon \Omega E({\mathcal O}X )\rightarrow E(X)$ of the above weak fibration is called the {\em coarse assembly map} associated to $E$.
\end{definition}

It is shown in \cite{Mitch13} that the map $\partial$ is a weak equivalence if and only if the space $E({\mathcal C}X)$ is weakly contractible.  If the space $X$ is flasquje, then so is the closed cone ${\mathcal C}X$.  Thus the map $\partial$ is a weak equivalence whenever $X$ is flasque.

Our aim in the rest of this section is to generalise this observation.

\begin{definition}
Let $X$ and $Y$ be coarse spaces.  Let $A\subseteq X$, and let $f\colon A\rightarrow Y$ be a coarse map.  Then we define the space obtained by {\em attaching} $X$ to $Y$ {\em along} $A$ by the map $f$ to be the quotient $X\cup_A Y = {X\vee Y}/{\sim}$, where $\sim$ is the equivalence relation defined by saying $a\sim f(a)$ whenever $a\in A$.
\end{definition}

Now, let $R$ be a generalised ray.  We define the {\em coarse $R$-disk} and {\em coarse $R$-sphere} of dimensions $n$ and $n-1$ respectively to be the spaces
\[ D_R^n = (R\vee R)^n \times R \qquad S_R^{n-1} = (R\vee R)^n \times \{ 0 \} . \]

Note that $S_R^n \subseteq D_R^n$.  It is shown in \cite{MiNS} that the coarse disk $D_R^n$ and generalised ray $R$ are coarsely homotopy-equivalent.\footnote{There is a similar result in \cite{Mitch4} for a more primitive notion of homotopy.}  Given a coarse map $f\colon S_R^n \rightarrow Y$, we can form the coarse space $D_R^n \cup_{S_R^n} Y$.

\begin{lemma} \label{attach}
Let $\pi \colon D_R^n \vee Y \rightarrow D_R^n \cup_{S_R^n}Y$ be the quotient map.  Then we have a coarsely excisive decomposition $D_R^n \cup_{S_R^n} Y= \pi [D_R^n ] \cup \pi [Y]$, and the spaces $\pi [D_R^n]$ and $\pi [Y]$ are coarsely equivalent to the spaces $D_R^n$ and $Y$ respectively.  Further, the intersection $\pi [D_R^n] \cap \pi  [Y]$ is coarsely equivalent to the coarse sphere $S_R^{n-1}$.
\end{lemma}

\begin{proof}
Apart from the comment about the intersection, this result follows in proposition 5.5 of \cite{Mitch4}, though it is not shown explicitly in \cite{Mitch4} that the spaces $\pi [D_R^n]$ and $D_R^n$ are coarsely equivalent.  So we do this step here.

Observe that the spaces $D^n_R$ and
\[ (D_R^n)' = \{ (x,t) \in (R\vee R)^n \times R \ |\ t\geq 1 \} \]
are coarsely equivalent.  Now, maps which preserve controlled sets, such as $\pi$, preserve coarsely equivalent spaces, so the spaces $\pi [D_R^n]$ and $\pi [(D_R^n)']$ are coarsely equivalent.

But by definition, the map $\pi |_{(D_R^n)'}  \colon (D_R^n)' \rightarrow D_R^n \cup_{S_R^n}Y$ is a coarse equivalence onto its image.  To summarise, we have a chain of coarse equivalences
\[ D_R^n \sim (D_R^n)' \sim \pi [ (D_R^n)' ] \sim \pi [D_R^n ]  . \]

As for the statement about intersections, the spaces $S_R^n = (R\vee R)^n \times \{ 0 \}$ and $(R\vee R)^n \times \{ 1 \}$ are certainly coarsely equivalent.  Therefore the images $\pi [S_R^n]$ and $\pi [(R\vee R)^n \times \{ 1 \}$ are coarsely equivalent. 

But the space $\pi [ (R\vee R)^n \times \{ 1\}$ is coarsely equivalent to $S_R^n$, and $\pi [S_R^n ]=  \pi [D_R^n ]\cap \pi [Y]$.  It follows that the spaces $S_R^n$ and $\pi [D_R^n ]\cap \pi [Y]$ are coarsely equivalent.
\end{proof}

The following also comes from \cite{Mitch4}.

\begin{definition} \label{CW}
We call a coarse space $X$ a {\em finite coarse $CW$-complex} if we have subsets
\[ X_0 \subseteq X_1 \subseteq X_2 \subseteq \cdots \subseteq X_n = X \]
where:

\begin{itemize}

\item The space $X_0$ is a finite disjoint union of generalised rays.

\item The space $X_n$ is obtained from $X_{n-1}$ by attaching finitely many coarse $n$-dimensional disks along coarse $(n-1)$-dimensional spheres.

\end{itemize}

\end{definition}

We now come to the promised major result of this section.

\begin{theorem} \label{ECW}
Let $E$ be a coarsely excisive functor, and let $X$ be a finite coarse $CW$-complex.  Then the coarse assembly map associated to $E$ is a weak equivalence.
\end{theorem}

\begin{proof}
We will show that $E({\mathcal C}X)$ is weakly contractible.

Suppose we have a decomposition
\[ X_0 \subseteq X_1 \subseteq X_2 \subseteq \cdots \subseteq X_n = X \]
as in the above definition.

First, we know that a generalised ray is flasque, and it is easy to check that a finite coarse disjoint union of flasque spaces is flasque.  Hence $X_0$ is flasque.  By lemma \ref{M13}, the closed cone of a flasque space is flasque, so by definition of an excisive functor, $E({\mathcal C}X_0)$ is contractible.

We now work by induction.  It suffices to show that if $E({\mathcal C}Y)$ is weakly contractible, and we have a map $f\colon S_R^{k-1} \rightarrow Y$, then the spectrum $E({\mathcal C}(D_R^k \cup_{S_R^{k-1}Y} Y))$ is also weakly contractible.

By lemma \ref{attach} and the homotopy pushout axiom in the definition of a coarsely excisive functor, we have a long exact sequence of stable homotopy groups
\[ \pi_n E({\mathcal C}D_R^k ) \oplus \pi_n E({\mathcal C}Y ) \rightarrow \pi_n E({\mathcal C} (D_R^k \cup_{S_R^{k-1}Y} Y)) \rightarrow \pi_{n-1} E({\mathcal C} S_R^{k-1}) \rightarrow \pi_{n-1} E({\mathcal  C}D_R^k ) \oplus \pi_{n-1} E({\mathcal C}Y ) . \]

Now we know that $E({\mathcal C}Y)$ is weakly contractible.  The coarse space $D_R^k$ is coarsely homotopy-equivalent to the flasque space $R$, so by lemma \ref{M13}, the space $E({\mathcal C}D_R^k)$ is weakly contractible.  It therefore suffices to prove that the spectrum $E({\mathcal C}S_R^m)$ is weakly contractible for each $m\in \N$.

The space $S_R^0$ is a coarse disjoint union of two generalised rays, so $E({\mathcal C}S_R^0)$ is weakly contractible as argued above.  We now work again by induction.  Suppose the spectrum $E({\mathcal C}S_R^p)$ is weakly contractible; we need to show that $E({\mathcal C}S_R^{p+1})$ is weakly contractible.

Observe that we have a coarsely excisive decomposition $S_R^{p+1} = A\cup B$ where $A$ and $B$ are both coarsely equivalent to the coarse disk $D_R^{p+1}$, and $A\cap B$ is coarsely equivalent to the coarse sphere one dimension lower, $S_R^p$.  Then it follows from an argument similar to the exact sequence argument three paragraphs above that the spectrum $E(S_R^{p+1})$ is weakly contractible, and we are done.
\end{proof}

\section{Bivariant Algebraic $K$-theory}

Let $R$ be a commutative ring with identity, and let $A$ be an $R$-algebra.  Let $A[x]$ be the algebra of polynomials over $A$ in one variable.  Define algebra homomorphisms $\partial_0, \partial_1 \colon A[x]\rightarrow A$ by writing $\partial_i (f) = f(i)$.

As in \cite{Ger}, given algebra homomorphisms $\alpha ,\beta \colon A\rightarrow B$, we call an algebra homomorphism $H \colon A\rightarrow A\rightarrow B[x]$ an {\em elementary algebraic homotopy} from $\alpha$ to $\beta$ if $\partial_0 \circ H = \alpha$ and $\partial_1 \circ H = \beta$.  More generally, we call $\alpha$ and $\beta$ {\em algebraically homotopic} if they can be linked by a finite chain of algebraic homotopies, and write $\alpha \simeq \beta$.

Algebraic $K$-theory is {\em not} invariant under algebraic homotopies.  However, in \cite{We}, the definition of algebraic $K$-theory is modified to define a series of groups, $KH_n (A)$, called {\em homotopy algebraic $K$-theory groups}.

In another world, that of complex $C^\ast$-algebras, there is a well-established notion of {\em bivariant $K$-theory groups}, $KK_n (A,B)$, defined for $C^\ast$-algebras $A$ and $B$; see \cite{Kas3}, or \cite{Hig3,Ska} for overviews.  These bivariant $K$-theory groups generalise ordinary $K$-theory in that $KK_n (\C , A) = K_n (A)$ for any $C^\ast$-algebra $A$.

Now, in \cite{CoT}, and independently in \cite{Gark1}, there are constructions of {\em bivariant algebraic $K$-theory} groups $kk_n (A,B)$ for $R$-algebras $A$ and $B$.  These groups have certain universal properties that ensure they are isomorphic.\footnote{For the record, we are using the matrix-stable theory here from \cite{Gark1}; we need matrix stability later on in this article.}  The article \cite{Gark2} represents this bivariant $K$-theory by spectra.  Specifically, we have the following result.

\begin{theorem}
Let $\mathcal R$ be the category of $R$-algebras.  Let $\Sp$ be the category of spectra. We have a functor $\KK \colon {\mathcal R}^\mathrm{op} \times {\mathcal R} \rightarrow \Sp$ with the following properties.

\begin{itemize}

\item Let $A$, $B$ and $C$ be $R$-algebras.  Let $\alpha ,\beta \colon B\rightarrow C$ be algebraically homotopic.  Then the induced maps $\alpha_\ast , \beta_\ast \colon \KK (A,B)\rightarrow \KK (A,C)$ are equal, and the induced maps $\alpha^\ast , \beta^\ast \colon \KK (C,A)\rightarrow \KK (C,B)$ are equal.

\item  Let 
\[ 0 \rightarrow A\stackrel{i}{\rightarrow} B\stackrel{j}{\rightarrow} C\rightarrow 0 \]
be a split exact short exact sequence of $R$-algebras.  Let $D$ be another $R$-algebra.  Then we have induced homotopy fibrations of spectra
\[ \KK (D,A)\stackrel{i_\ast}{\rightarrow} \KK (D,B)\stackrel{j_\ast}{\rightarrow} \KK (D,C)\rightarrow 0 \]
and
\[ \KK (C,D) \stackrel{j^\ast}{\rightarrow} \KK (B,D) \stackrel{i^\ast}{\rightarrow} \KK (A,D) \rightarrow 0 . \]

\item Let $A$ and $B$ be $R$-algebras.  Let $M_n (A)$ be the algebra of $n\times n$ matrices with values in $A$, and let $M_\infty (A)$ be the direct limit of the sequence of matrix algebras $(M_n(A))$ under the top-left inclusions $x\mapsto \left( \begin{array}{cc}
x & 0 \\
0 & 0 \\
\end{array} \right)$.  Then the homomorphism $A\rightarrow M_\infty (A)$ defined by top-left inclusion induces weak equivalences of spectra $\KK (M_n(A),B) \rightarrow \KK (A,B)$ and $\KK (B,A) \rightarrow \KK (B,M_n (A))$.

\item Let $A$ be an $R$-algebra.  Then we have a natural isomorphism $\pi_n \KK (R,A)\cong KH_n (A)$ for all $n$.

\end{itemize}

\noproof
\end{theorem}

We call an algebra homomorphism $\alpha \colon A\rightarrow B$ an {\em stable algebraic homotopy equivalence} if there is an algebra homomorophism $\beta \colon B\otimes_R M_\infty (R) \rightarrow A \otimes_R M_\infty (R)$ such that $\beta \circ (\alpha \otimes \id_{M_\infty (R)}) \simeq \id_{M_\infty (A)}$ and $(\alpha \otimes \id_{M_\infty (R)}) \circ \beta \simeq \id_{M_\infty (B)}$.  By the above theorem, given an $R$-algebra $C$, a stable algebraic homotopy equivalence $\alpha \colon A\rightarrow B$ induces an equivalence of spectra $\alpha_\ast \colon \KH (A)\rightarrow \KH (B)$.

Now, we would like a version of homotopy algebraic $K$-theory associated to a coarse space and a fixed $R$-algebra, $D$, creating a coarsely excisive functor with 'coefficients' in the spectrum $\KK (D,R)$.  In order to do this, we need to extend the above from algebras to algebroids.

To be more precise about what we need, given a ring $R$, recall (see \cite{Mi}) that a category $\mathcal A$ is termed a {\em unital $R$-algebroid} if each morphism set $\Hom
(a,b)_{\mathcal A}$ is a left $R$-module, and composition of
morphisms
\[ \Hom (b,c)_{\mathcal A}\times \Hom (a,b)_{\mathcal A}\rightarrow
\Hom (a,c)_{\mathcal A} \]
is $R$-bilinear.

Non-unital $R$-algebroids are defined similarly, but we drop the requirement that identity morphisms $1\in \Hom (a,a)_{\mathcal A}$ have to exist.  

A {\em unital algebroid homomorphism} between $R$-algebroids $\mathcal A$ and $\mathcal B$ is simply a functor $\phi \colon {\mathcal A}\rightarrow {\mathcal B}$ that is linear on each morphism set.  Non-unital algebroid homomorphisms are defined similarly, but we drop the condition $\phi (1)=1$ from the definition of a functor.

Given objects, $a$ and $b$ in an $R$-algebroid $\mathcal A$, an object $a\oplus b$ is called a
{\em biproduct} of the objects $a$ and $b$ if it comes equipped with
morphisms $i_a \colon a\rightarrow a\oplus b$, $i_b \colon
b\rightarrow a\oplus b$, $p_a \colon a\oplus b\rightarrow a$, and
$p_b \colon a\oplus b \rightarrow b$ satisfying the equations
$$p_a i_a = 1_a \qquad p_b i_b = 1_b \qquad i_a p_a + i_b p_b = 1_{a\oplus b}$$

An $R$-algebroid $\mathcal A$ is called {\em additive} if every pair of objects has a biproduct, and we have a {\em zero object} $0\in \Ob ({\mathcal A})$ such that $a\oplus 0$ is isomorphic to $a$ for all $a\in \Ob ({\mathcal A})$.

Now, in \cite{Jo2}, Joachim defined the $K$-theory of $C^\ast$-categories, and in particular $K$-theory spectra, by use of a suitable functor from $C^\ast$-categories to $C^\ast$-algebras; $C^\ast$-algebra $K$-theory is of course well-known (see for instance \cite{RLL,W-O} for expositions).  We adapt Joachim's approach here.

Firstly, given an algebroid $\mathcal A$, note that we can define an algebra $S({\mathcal A})$ by writing
\[ S({\mathcal A}) = \oplus_{a,b\in \Ob ({\mathcal A})} \Hom (a,b)_{\mathcal A} .\]

Given $x\in \Hom (a,b)_{\mathcal A}$ and $y\in \Hom (c,d)_{\mathcal A}$, we define the product $xy$ to be the composition $x\circ y \in \Hom (c,b)_{\mathcal A}$ if $d=a$, and $0$ otherwise.

In some sense, the algebra $S({\mathcal A})$ obviously contains the information we want for $K$-theory purposes from the algebroid $\mathcal A$; this idea is made more concrete in \cite{Jo2} by looking at modules (and incidentally, Joachim's $K$-theory for $C^\ast$-categories is the same as that defined in \cite{Mitch2.5}).  However, it has the drawback of not being functorial; as in \cite{Jo2}, the construction needs to be modified.

\begin{definition}
Let $\mathcal A$ be an algebroid.  We define the {\em Joachim algebra} $F({\mathcal A})$ to be the free algebra generated by the morphisms of $\mathcal A$, with the relations
\[ r (x) + s (y) = (r x+s y) \qquad r,s \in R, x,y\in \Hom (a,b)_{\mathcal A}  , \]
and
\[ (x)(y) = (xy) \qquad x\in \Hom( b,c)_{\mathcal A} ,\ y\in \Hom (a,b)_{\mathcal A} .\]

Here, we write $(x)$ to denote the image in $F({\mathcal A})$ of a morphism, $x$, in the algebroid $\mathcal A$.
\end{definition}

Given an algebroid homomorphism $\phi \colon {\mathcal A}\rightarrow {\mathcal B}$, we have an induced algebra homomorphism $\phi_\ast \colon F({\mathcal A}) \rightarrow F({\mathcal B})$ defined by the obvious formula $\phi_\ast ((x)) = (\phi (x))$ for each morphism $x$ in $\mathcal A$.  With such induced morphisms, $F$ is a a functor.  Further, we have a natural algebroid homomorphism $\tau \colon {\mathcal A}\rightarrow F({\mathcal A})$ defined by writing $\tau (x) =(x)$ for any morphism $x$.

Further, we have a {\em universal property}, namely, given an algebroid homomorphism $\phi \colon {\mathcal A}\rightarrow B$ into an algebra $B$, there is a unique algebra homomorphism $\theta \colon F({\mathcal A})\rightarrow B$ such that $\phi = \theta \circ \tau$.

In particular, there is an obvious homomorphism $\eta \colon {\mathcal A}\rightarrow S({\mathcal A})$, and so a unique algebra homomorphism $\sigma \colon F({\mathcal A}) \rightarrow S({\mathcal A})$ such that $\eta = \sigma \circ \tau$.

Just as in \cite{Jo2}, we have the following result.

\begin{proposition} \label{FS}
The map $\sigma \colon F({\mathcal A})\rightarrow S({\mathcal A})$ is a stable algebraic homotopy equivalence.
\noproof
\end{proposition}

\begin{definition}
Let $\mathcal A$ and $\mathcal B$ be $R$-algebroids.  Then we define the {\em bivariant algebraic $K$-theory spectrum}
\[ \KK ({\mathcal A},{\mathcal B}) = \KH (F({\mathcal A}), F({\mathcal B})) . \]
\end{definition}

Let $\mathcal A$, $\mathcal B$ and $\mathcal C$ be $R$-algebroids with the same set of objects.  We call a sequence of algebroid homomorphisms
\[ 0 \rightarrow {\mathcal A}\stackrel{i}{\rightarrow} {\mathcal B} \stackrel{i}{\rightarrow} {\mathcal C} \rightarrow 0 \]
a {\em split short exact sequence} if $i$ and $j$ are the identity map on the set of objects, there is an algebroid homomorphism $k \colon {\mathcal C}\rightarrow {\mathcal B}$ such that $j \circ k = \id_{\mathcal C}$, and for all objects $a$ and $b$ we have a split exact sequence of abelian groups
\[ 0 \rightarrow \Hom (a,b)_{\mathcal A}\stackrel{i}{\rightarrow} \Hom (a,b)_{\mathcal B} \stackrel{j}{\rightarrow} \Hom (a,b)_{\mathcal C} \rightarrow 0 \]
with splitting given by $k$.

It is easy to check that the functor $F$ takes split short exact sequences of algebroids to split short exact sequences of algebras.  We therefore have the following result.

\begin{proposition} \label{propertyB}
Let 
\[ 0 \rightarrow {\mathcal A}\stackrel{i}{\rightarrow} {\mathcal B}\stackrel{j}{\rightarrow} {\mathcal C}\rightarrow 0 \]
be a split exact short exact sequence of $R$-algebroids.   Let $D$ be a fixed $R$-algebra.  Then we have induced homotopy fibrations of spectra
\[ \KK (D,{\mathcal A})\stackrel{i_\ast}{\rightarrow} \KK (D,{\mathcal B})\stackrel{j_\ast}{\rightarrow} \KK (D,{\mathcal C}) \]
and
\[ \KK ({\mathcal C},D)\stackrel{j^\ast}{\rightarrow} \KK ({\mathcal B},D)\stackrel{i^\ast}{\rightarrow} \KK ({\mathcal A},D)  . \]
\noproof
\end{proposition}

We call an algebroid homomorphism $\phi \colon {\mathcal A}\rightarrow {\mathcal B}$ an {\em algeboid equivalence} if there is an algebroid homomorphism $\psi \colon {\mathcal B}\rightarrow {\mathcal A}$, along with natural isomorphisms $X \colon \psi \circ \phi \rightarrow \id_{\mathcal A}$ and $Y\colon \phi \circ \psi \rightarrow \id_{\mathcal B}$

\begin{proposition} \label{propertyA}
Let $\phi \colon {\mathcal A}\rightarrow {\mathcal B}$ be an algebroid equivalence.  Let $D$ be an $R$-algebra.  Then the induced map $\phi_\ast \colon \KK (D,{\mathcal A})\rightarrow \KK (D,{\mathcal B})$ is a weak equivalence.
\end{proposition}

\begin{proof}
It is clear that the algebras $S({\mathcal A})$ and $S({\mathcal B})$ are stably algebraically homotopy equivalent.  It follows by proposition \ref{FS} that the map $\phi_\ast \colon F({\mathcal A})\rightarrow F({\mathcal B})$ is a stable homotopy equivalence.  The result now follows by stable homotopy invariance of bivariant algebraic $K$-theory.
\end{proof}

This innocuous-seeming result is vital to us when we look at assembly maps, and, incidentally, is the point where we need the matrix stability of bivariant algebraic $K$-theory.

\section{Equivariant Assembly}

Let $G$ be a discrete group.  A coarse space equipped with a right $G$-action is termed a {\em coarse $G$-space}.

We call a subset, $A$, of a coarse $G$-space $X$ {\em cobounded} if there is a bounded subset $B\subseteq X$ such that $A\subseteq BG$.  The {\em coarse $G$-category} is the category where the objects are coarse $G$-spaces, and the morphisms are controlled equivariant maps where the inverse image of a cobounded set is cobounded.

As in \cite{Mitch13}, the notions of coarse homotopy and flasqueness generalise in an obvious way to the equivariant notions of {\em coarse $G$-homotopy} and {\em $G$-flasqueness}.

\begin{definition}
We call a functor $E_G$ from the coarse $G$-category to the category
of spectra {\em coarsely $G$-excisive} if the following conditions
hold.

\begin{itemize}

\item The spectrum $E_G(X)$ is weakly contractible whenever the coarse
space $X$ is $G$-flasque.

\item The functor $E_G$ takes coarse $G$-homotopy equivalences to weak homotopy
equivalences of spectra.

\item Given a coarsely excisive decomposition $X=A\cup B$, where $A$ and $B$
are coarse $G$-spaces, we have a homotopy push-out diagram
\[ \begin{array}{ccc}
E_G(A\cap B) & \rightarrow & E_G(A) \\
\downarrow & & \downarrow \\
E_G(B) & \rightarrow & E_G(X) \\
\end{array} . \]

\item Let $X$ be a cobounded coarse $G$-space.  Then the constant map $c\colon X\rightarrow +$ induces a stable
equivalence $c_\ast \colon E_G(X)\rightarrow E_G(+)$.

\end{itemize}

\end{definition}

As in the non-equivariant case, if $E_G$ is coarsely $G$-excisive, we have a weak fibration
\[ E_G(X) \stackrel{j}{\rightarrow} E_G({\mathcal C}X) \stackrel{v_G}{\rightarrow} E_G({\mathcal O}X) \]
and associated boundary map
\[ \partial_G \colon \Omega E_G({\mathcal O}X ) \rightarrow E_G(X) .  \]

This map is termed the {\em equivariant assembly map} associated to the functor $E_G$.

\begin{definition}
Let $E$ be a coarsely excisive functor.  Then we say a coarsely $G$-excisive functor $E_G$ has the {\em local property} relative to $E$ if there is a natural
map $i \colon E_G(X)\rightarrow E(X)$, such that if $X={\mathcal O}Y$, where $Y$ is a free coarse cocompact $G$-space, and $\pi \colon X\rightarrow X/G$ is the quotient map, then the composite
\[ \pi_\ast \circ i = i \circ \pi_\ast \colon E_G (X)\rightarrow E(X/G) \]
is a stable equivalence.
\end{definition}

In its most general terms, the main result of \cite{Mitch13}, which we apply to homotopy algebraic $K$-theory in the next section, is the following.  We call it the {\em descent theorem}.

\begin{theorem} \label{DT}
Let $E_G$ be a coarsely $G$-excisive functor. Let $E$ be a coarsely excisive functor with the local property relative to $E_G$.   Let $X$ be a free coarse $G$-space, that is, as a topological space, $G$-homotopy equivalent to a finite $G$-$CW$-complex.

Suppose the coarse assembly map for the functor $E$ and space $X$ is a weak equivalence.  Then the map $\partial_G \colon \Omega E_G ({\mathcal O}X) \rightarrow E_G (X)$ is injective at the level of stable homotopy groups.
\noproof
\end{theorem}

\section{Homology}

Our plan in this section is to define a coarsely excisive functor with 'coefficients' in algebraic $KK$-theory.  As in \cite{Mitch13}, we adapt the approach taken for controlled algebraic $K$-theory in such articles as \cite{ACFP,Bart,CP,Weiss}.

\begin{definition}
Let $X$ be a coarse space, and let $\mathcal A$ be an additive $R$-algebroid.  Let $B(X)$ be the collection of bounded subsets of $X$, partially ordered by inclusion.  Note that a partially ordered set can be regarded as a category; a {\em geometric $\mathcal A$-module} over $X$ is a functor, $M$, from the collection of bounded subsets of $X$\footnote{Regarded as a category by looking at the usual partial ordering.} to the category $\mathcal A$ such that for any bounded set $B$ the natural map
\[ \oplus_{x\in B} M(\{ x \}) \rightarrow M(B) \]
induced by the various inclusions is an isomorphism, and the {\em support}
$$\supp (M) = \{ x\in X |\ M(\{ x\}) \neq 0 \}$$
has finite intersection with every bounded subset of $X$.
\end{definition}

We call a subset $S\subseteq X$ {\em locally finite} if $S\cap B$ is finite whenever $B$ is bounded.  Thus the second of the above conditions says simply that $\supp (M)$ is locally finite.

A {\em morphism} $\phi \colon M\rightarrow N$ between geometric $\mathcal A$-modules over $X$ is a collection of
morphisms $\phi_{x,y} \colon M_y \rightarrow N_x$ in the algebroid
$\mathcal A$ such that for each fixed point $x\in X$, the morphism $\phi_{x,y}$ is non-zero for only finitely many points
$y\in X$, and for each fixed point $y\in X$, the morphism $\phi_{x,y}$ is non-zero for only finitely many points $x\in X$.

Composition of morphisms $\phi \colon M\rightarrow N$ and $\psi \colon N\rightarrow P$ is defined by the formula
\[ (\psi \circ \phi )_{x,y} (\eta ) = \sum_{z\in X} \psi_{x,z} \circ \phi_{z,y} (\eta ) . \]

We define the {\em support} of a morphism $\phi$
\[ \supp (\phi ) = \{ (x,y)\in X \ |\ \phi_{x,y}\neq 0 \} . \]

\begin{definition}
The category ${\mathcal A}[X]$ consists of all geometric
$\mathcal A$-modules over $X$ and morphisms such that the support is
controlled with respect to the coarse structure of $X$.
\end{definition}

Observe that ${\mathcal A}[X]$ is again an additive $R$-algebroid.  Given a coarse map $f\colon X\rightarrow Y$, and a geometric $\mathcal A$-module, $M$, over $X$, we have a geometric $A$-module $f_\ast [M]$ defined by writing $f_\ast [M](B) = M(f^{-1}[B])$ for each bounded set $B\subseteq Y$.  Given a morphism $\phi \colon M\rightarrow N$ be a morphism in the category ${\mathcal A}[X]$, we have an induced morphism $f_\ast [\phi ] \colon f_\ast [M]\rightarrow f_\ast [N]$ given by the formula
\[ f_\ast [\phi ]_{y_1,y_2} = \sum_{x_1\in f^{-1}(y_1) \atop x_2\in f^{-1}(y_2)} \phi_{x_1,x_2} .\]

The above turns the assignment $X\mapsto {\mathcal A}[X]$ into a functor from the coarse category to the category of small $R$-algebroids and algebroid homomorphisms.

Now, the following is proved in exactly the same way as theorem 6.13 from \cite{Mitch13}.  Proposition \ref{propertyA} is needed to establish coarse homotopy invariance, and proposition \ref{propertyB} is needed for homotopy push-outs.

\begin{theorem} \label{Ahom}
Let $D$ be an $R$-algebra, and let $\mathcal A$ be an additive $R$-algebroid.  Then the functor $X\mapsto \KK (D,{\mathcal A}[X])$ is coarsely excisive.
\noproof
\end{theorem}

We also have an equivariant version of the construction.  To be more precise, let  $X$ be a coarse $G$-space, let $R$ be a ring, and let $\mathcal A$ be an additive $R$-algebroid.  Then we call a geometric $\mathcal A$-module, $M$, over $X$ {\em $G$-invariant} if $M_{xg} = M_x$ for all $x\in X$ and $g\in G$. A morphism $\phi \colon M\rightarrow N$ between such modules is termed {\em $G$-invariant} if $\phi_{xg,yg} = \phi_{x,y}$ for all $x,y\in X$.

\begin{definition}
We write ${\mathcal A}_G [X]$ to denote the category of $G$-invariant geometric $\mathcal A$-modules over $X$, and $G$-invariant morphisms.
\end{definition}

The following result is similar to theorem \ref{Ahom}.

\begin{theorem}
The assignment $X\mapsto \KK( D,{\mathcal A}_G [X])$ is a coarsely $G$-excisive functor.  Further, the functor $X\mapsto \KK (D,{\mathcal A}_G[X])$ has the local property relative to the functor $\KK (D,{\mathcal A} [X])$.
\noproof
\end{theorem}

Now, we turn our attention to the associated isomorphism conjecture; the $KH$-isomorphism conjecture, as formulated and discussed in \cite{BL}, is a special case.  

Fix an $R$-algebra $D$, an $R$-algebroid $\mathcal A$, and consider the functor defined on cocompact topological spaces (with a coarse structure making any such space cobounded) by the formula
\[ F_G (X) = \Omega \KK (D,{\mathcal A}_G [{\mathcal O}X]) . \]

Then the {\em $KK$-assembly map} with coefficients in $\mathcal A$ is defined to be the map $c\colon F_G (X)\rightarrow F_G( +)$
induced by the constant map $X\rightarrow +$.  Exactly as in theorem 8.17 of \cite{Mitch13}, this map amounts to the same thing (up to weak equivalence) as the equivariant assembly map
\[ \partial_G \colon \Omega \KK (D, {\mathcal A}_G [{\mathcal O}X]) \rightarrow \KK (D, {\mathcal A}_G[X]  ) \]
whenever $G$ acts cocompactly on $X$.

Now, let $\underline{E}G$ be the {\em classifying space for proper actions} of $G$.  As in \cite{DL}, we describe $\underline{E}G$ as a unique (up to $G$-homotopy) $G$-$CW$-complex $E(G,{\mathcal F})$, with the property that, for as subgroup $H$, the fixed point set $\underline{E}G^H$ is empty if $H$ is infinite, and $G$-contractible if $H$ is finite.

\begin{definition}
The {\em $KK$-isomorphism conjecture} for the group $G$ with coefficients in $D$ and $\mathcal A$ asserts that the above $KK$-assembly map is a weak equivalence when $X=\underline{E}G$.
\end{definition}

Now, the descent theorem can be applied to tell us things about injectivity of the $KK$-assembly map.  To be precise, by theorem \ref{DT} and the above, we have the following result.

\begin{theorem} \label{DT}
Suppose that the space $EG$ is $G$-homotopy equivalent to a finite $G$-$CW$-complex.  Suppose we have a coarse structure on $EG$ where the $G$-action on $EG$ is by coarse maps, and the coarse assembly map $\partial \colon \KK (D,{\mathcal A} [{\mathcal O}X]) \rightarrow \KK (D, {\mathcal A}[X] $ is a weak equivalence.  Then the $KK$-assembly map is injective at the level of stable homotopy groups for the space $EG$.
\noproof
\end{theorem}

By theorem \ref{ECW}, we have the following corollary.

\begin{corollary}
Suppose that the space $EG$ is $G$-homotopy equivalent to a finite $G$-$CW$-complex.  Suppose we have a coarse structure on $EG$ where the $G$-action on $EG$ is by coarse maps, and $EG$ is coarsely homotopy equivalent to a finite coarse $CW$-complex.  Then the $KK$-assembly map is injective at the level of stable homotopy groups for the space $EG$.
\noproof
\end{corollary}

\bibliographystyle{plain}

\bibliography{data}

\end{document}